\documentclass[10pt]{article}   		
\usepackage{amssymb}
\usepackage{amsmath}
\usepackage{enumerate}
\usepackage{graphicx}
\usepackage{psfrag}
\usepackage{color} 
\usepackage{comment}

\def\ddefine#1{{\bf #1}}
\def\vp{\vspace{\parskip}}
\def\mpp#1{\marginpar{\tiny \%\%  #1}}
\def\mpp#1{{}}

\newtheorem{thm}{Theorem}
\newtheorem{lem}[thm]{Lemma}
\newtheorem{prop}[thm]{Proposition}
\newtheorem{definition}[thm]{Definition}
\newtheorem{cor}[thm]{Corollary}

\newenvironment{proof}{{\em Proof }} {\hfill{$\Box$}}

\def\Nn{{\mathbb N}}
\def\Zz{{\mathbb Z}}
\def\Rr{{\mathbb R}}
\def\Hh{{\mathbb H}}
\def\Qq{{\mathbb Q}}

\def\aA{{\mathcal A}} 
\def\bB{{\mathcal B}}
\def\aAb{{\mathcal A}_{\mathcal{B}}}
\def\aAO{{{\mathcal A}_0}}

\def\tX{{\tt Y}}
\def\tY{{\tt W}}

\def\wlen#1{{\| #1\|}}

\def\approxN{ \underset N{\approx} }

\def\vect#1{{\bf #1}}
\def\vectv{{\vect v}}
\def\vA{{\vect A}}
\def\normed#1{{\overline{#1}}}

\def\To{\rightarrow}
\def\anc{\textrm{\sf P}}
\def\anct{\textrm{\sf Q}}

\def\fF{{\mathcal F}} 
\def\vertices{{V}} 
\def\labelmap{{\omega}} 
\def\orbit{{\mathcal O}} 
\def\Gpq{{\Gamma_{p,q}}} 
\def\Gsurf{{\Gamma_{4g,4g}}} 
\def\reducedorbitgraph{{\overline{\Gamma_\orbit}}} 
\def\Asharp{{\mathcal{A}_\#}}
\def\sigmasharp{{\sigma_\#}}
\def\XF{{X_\mathcal{F}}}
\def\fraks{{\mathfrak{s}}}
\def\row{{\mathrm{row}}}
\def\dy{{dy}}
\def\HG{{\mathcal{H}\Gamma}}
\def\Parent{{\mathrm{Par}}}
\def\Deltapq{{\Delta_{p,q}}}

\parskip = \baselineskip

\def\vp{\vspace{\baselineskip}}

\title{Strongly aperiodic subshifts on surface groups}
\author{David. B. Cohen\\ Univ. Chicago \\ {\tt davidbrucecohen@gmail.com} \and C. Goodman-Strauss \\ Univ. Arkansas \\ {\tt strauss@uark.edu}}
\date{}

\begin{document}

\maketitle

\abstract{We give strongly aperiodic subshifts of finite type on every hyperbolic surface group; more generally, for each pair of expansive primitive symbolic substitution systems with incommensurate growth rates, we construct strongly aperiodic subshifts of finite type on their orbit graphs.}

\section{Introduction}

A ``subshift of finite type" on a group $G$, with some alphabet $\aA$, is a subset of $\aA^G$ defined by allowed (or forbidden) local patterns, and is ``strongly aperiodic" if every element of it  has trivial stabilizer. 

Strongly aperiodic subshifts of finite type (SASFTs), and the close analog of strongly aperiodic tilings, have been studied in a variety of contexts, notably $\Zz^n$ (\cite{berger, robinson} and a great many others) $\Hh^n$ \cite{gs_sasht,kari_undH}, the integer Heisenberg group \cite{ssu},  higher-rank symmetric spaces \cite{mozes_stap}, polycyclic groups \cite{jeandel_poly} and $\Zz\times G$ for a general class of group $G$ \cite{jeandel_stap}. Whether or not a group admits  a SASFT is a 
quasi-isometry invariant under mild conditions \cite{cohen}, and a commensurability invariant \cite{carroll_penland}. It is known that no free group admits a SASFT \cite{piantadosi} and more generally, nor does any group with two or more ends  \cite{cohen}, nor any group with undecidable word problem \cite{jeandel_stap}.

Here we construct a strongly aperiodic subshift of finite type on genus $>1$ surface groups (Corollary~\ref{Cor:MainThm}). Though strongly aperiodic tilings have been constructed in $\Hh^2$~\cite{gs_stap, gs_sasht, kari_undH} the underlying tiles do not admit a tiling of any compact fundamental domain. Here we use an observation from~\cite{gs_rptt}, that any regular tiling of $\Hh^2$  can be encoded as an ``orbit graph" of an expansive primitive symbolic substitution system (greatly distorting the geometry but maintaining the combinatorial structure).  In turn,  given any pair of  expansive primitive symbolic substitution systems with incommensurate growth rates, we construct SASFT's on their orbit graphs. 

This technique can be seen as a generalization of the construction in~\cite{gs_stap}, in turn derived from~\cite{kari}--- in our terms here, the underlying substitution systems there are simply ${\tt 0}\mapsto{\tt 0}^2$ and ${\tt 0}\mapsto{\tt 0}^3$. 

\begin{figure}
\centerline{\includegraphics{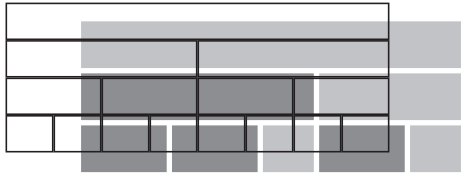}}
\caption{A pair of overlaid pieces of orbit tilings --- duals of orbit graphs for the substitution systems ${\tt 0}\mapsto{\tt 00}$ (with growth rate 2) and 
${\tt A}\mapsto{\tt AB}$, ${\tt B}\mapsto{\tt AAB}$ (with growth rate $\frac 1 2 (3+\sqrt 5)$). 
In effect, we create a set of tiles encoding the local combinatorics of how these tilings may meet; any tiling by these tiles must enforce the  growth rates of the underlying substitution systems. In turn, as these growth rates are incommensurate, there can be no vertical period.}
\end{figure}

\section{Preliminary definitions}

We take the convention that $0$ is not a natural number, that $\Nn = \{1,2,\ldots, \}$.  We will write $n'$ for $n+1$. For any sequence $\{x_i\}\in\Rr^\Zz$, we define $$\textstyle \overset{n}{\underset i \sum}\ x_i:=
\left\{\begin{array}{clrlr} 0 & \textrm{for } n= 0\\
x_0+\ldots + x_{n-1} & \textrm{for } n >0\\
-x_{n}-\ldots - x_{-1} & \textrm{for } n <0\end{array}\right.$$
so that for all $n\in\Zz$, $x_n=\overset{n'}{\underset i \sum}\  x_i -\overset{n}{\underset i \sum}\  x_i$.

Let $\aA$ be any finite alphabet and $L\subset\aA^{*}$ be any  language on $\aA$. 
For any word $w\in\aA^\ast$, let $\wlen{w}$ be the 
length of $w$. We define the language 
$ L^{\infty}\subset\aA^{\Zz}$   of \ddefine {infinite 
words} to be sequences $\omega\in\aA^{\Zz}$ such that every finite 
subsequence $\omega(i)\ldots\omega(j)$ is a subword of some word in 
$L$. If $L$ is an 
infinite  
regular language, then by the Pumping Lemma,
 $ L^{\infty}\neq\emptyset$.
To compress notation, for finite or infinite words, we will often write $\omega_i$ for $\omega(i)$, and  $\omega_{(i\ldots j)}$ or $\omega(i\ldots j)$ for the word $\omega(i)\ldots\omega(j)$. For finite words we will also write $\omega_{(k\ldots)}$ or $\omega(k\ldots)$ for $\omega(k)\ldots \omega(\wlen{\omega})$.

Given an infinite set 
$\{u^{n}\}\subset(\aA^{*})^\Zz$ 
the    \ddefine{infinite 
concatenation} $\omega=\ldots u^{-1}u^{0}u^{1}\ldots \in \aA^\Zz$  satisfies,  for 
all $n\in\Zz$,
$\omega{((s_n)'\ldots s_{n'})}=u^{n}$ where $s_n=\overset{n}{\underset i \sum}\ {\wlen{u^i}}$.
 This definition coincides with what 
one might expect, taking $\omega(1)$ to coincide with $u^0(1)$. Note if $L$ is closed under concatenation, then every infinite concatenation of words in $L$ is in $L^\infty$. 
Given, say, alphabets $\aA$ and $\bB$, and any map $f:\aA\To\bB^\ast$, we will naturally define $f:\aA^{\ast}\To\bB^{\ast}$ as $f(w)=f(w_1)\ldots f(w_n)$, and $f:\aA^\Zz\To\bB^\Zz$ as $f(\omega) = \ldots f(\omega_{-1})f(\omega_0)f(\omega_1)\ldots $.

\paragraph{Productions}\ 

 A  \ddefine{production system}  $(\aA,L,R)$ is specified by an alphabet $\aA$,  language $L$ on $\aA$ and ``production rules", a finite subset $R$ of $\aA\times L$. A \ddefine{regular} production system is one for which $L$ is a regular language. (\cite{gs_rptt} gives many applications and examples.)

Given $(\aA,L,R)$, we extend our production rules $R$ to  define \ddefine{production relations} $R\subset (L\times L)\cup (L^\infty\times L^\infty)$ 
on finite and infinite words: For finite words $u,v\in L$,  $(u,v)\in R$ if and only if $v=v^1\ldots v^{\wlen{u}}$ 
for some $\{v^i\}\in L^{\wlen{u}}$ with each $(u_i,v^i)\in R$. 

For infinite words  $\omega,\sigma\in L^{\infty}$, $(\omega,\sigma)\in R$ if and only if  for some monotonic and onto ``parent function" $\anc:\Zz\To\Zz$, each $(\omega_i,\sigma|_{\anc^{-1}(i)})\in R$.  In other words, each letter $\omega_i$ in $\omega$ produces the word with indices in $\anc^{-1}(i)$ in $\sigma$ (formally, the function $\sigma:\anc^{-1}(i)\To\aA$), and each letter $\sigma_j$ in $\sigma$ appears in the word produced by $\omega_{\anc(j)}$. As $\anc$ is monotonic, each $\sigma^i = \sigma|_{\anc^{-1}(i)}$ {\em is} a well-defined  word, and as $\anc$ is onto $\sigma=\ldots \sigma^{-1}\sigma^{0}\sigma^{1}\ldots$  each $\sigma^i$ is of finite length. We say that ``$\omega$ produces $\sigma$ with respect to $\anc$." 

\vp
An \ddefine{orbit} in a production system $(\aA,L, R)$ is 
any set $\{(\omega^{i},\anc_{i})\}\in (L^{\infty}\times \Zz^\Zz)^\Zz$  such that for all 
$i\in\Zz$,   $\anc_i$ is monotonic and onto, and  $\omega^i$ produces $\omega^{i'}$ with respect to $\anc_{i}$.  ({\em A priori} orbits may or may not exist.)
An orbit is  \ddefine{periodic} if and only if there is some $\pi\geq 1$ with 
$\omega^{i}=\omega^{i+\pi}$, $\anc_{i}=\anc_{(i+\pi)}$ for all $i$; the period of the orbit is 
the minimal such $\pi$. 
 
 A production system is  \ddefine{expansive} if for every sequence of words $u^1,\ldots \in L$ with $(u^i,u^{i'})\in R$, there is some $n\in\Nn$, such that for all $i>n$,  $\wlen{u^{i}}>\wlen{u^1}$.

\vp

\paragraph{Symbolic substitutions}\ 
 
 A \ddefine{symbolic  substitution system} is a production system $(\aA, \aA^\ast, \sigma)$ such that $\sigma$ is a function $\aA\to\aA^\ast$ --- that is, for each $a\in\aA$, we have exactly one value $\sigma(a)$. We will write $(\aA, \sigma)$ for $(\aA, \aA^\ast, \sigma)$. A symbolic substitution system  is  \ddefine{primitive}  if and only if, for some $N$, for every $n>N$ and $a,b\in\aA$, the letter $b$ occurs within $\sigma^n(a)$. Note that a primitive symbolic substitution system is expansive if for at least some $a\in\aA$, $\wlen{\sigma(a)}>1$.

 It is well known that orbits exist in any primitive symbolic substitution systems, using that some letter $a$ appears in the interior of some $\sigma^n(a)$. In fact, there are uncountably many orbits, as some letter must appear more than once in some $\sigma^n(a)$, giving a countable sequence of choices in the construction of an orbit; however there are only countably many periodic orbits~\cite{gs_rptt}.

\vp

For functions $f,g:\Nn\to\Rr$, we write $f(k)=\Theta(g(k))$ to mean that for some constant $C>1$, for some $M$, for all $k>M$, we have $C^{-1}g(k)\leq f(k)\leq Cg(k)$.  Note that if $s^k=\Theta(t^k)$, then $s=t$.
For a given $\vectv=(v_1,
\ldots,v_n)\in\Rr^n$, let $|\vectv|=\sum |v_i|$  (that is, the 1-norm of $\vectv$) and let $\normed{\vectv}:=\vectv/|\vectv|$. 
The following is well-known (cf. \cite{kitchens}):
\begin{thm}\label{ThmEigenSubst} Let $(\aA,\sigma)$ be a primitive, expansive symbolic substitution system on a finite alphabet $\aA=\{a_1,\ldots, a_n\}$. For each $w\in\aA^{\ast}$, let $\vectv_w\in\Rr^n$ count each letter: that is, the $i$th coordinate of  $\vectv_w$ is the number of occurrences of $a_i$ in $w$.
Then there exists $\lambda>1$ such that 
for all $w\in\aA^*$, $\wlen{\sigma^{k}(w)}=\Theta(\lambda^k)$
.
\end{thm}

\begin{proof} Let $\vA$ be an $n\times n$ matrix --- a ``substitution matrix" for $\sigma$ ---  with each entry $(A_{ij})$, the number of occurrences of $a_i$ in $\sigma(a_j)$ (and thus non-negative). 

For a given word $w\in\aA^\ast$, note that $\vA \vectv_w = \vectv_{\sigma(w)}$, and so for each $k\in\Nn$, $\vA^k \vectv_w = \vectv_{\sigma^k(w)}$. 
Since $(\sigma,\aA)$ is primitive, there exists some $N\in\Nn$ such that for all $k>N$,  $\vA^k$ has all positive integer entries. By the Perron-Frobenius theorem,  $\vA$ has a real eigenvalue $\lambda>1$, that is strictly larger than the absolute value of any other eigenvalue.
Consequently, for any word $w\in\aA^\ast$, 
$\wlen{\sigma^{k}(w)}= |\vectv_{\sigma^{k}(w)}| = |\vA^k\vectv_w|=
 \Theta(\lambda^k \wlen{w})=\Theta(\lambda^k)$.  \end{proof}

\mpp{``distribution" is not a standard term; couldn't think of a better one.---Cgs}
\vp We call  $\lambda$ the \ddefine{growth rate} of $(\aA,\sigma)$. A \ddefine{distribution} of $(\aA,\sigma)$ is any left eigenvector $\nu\in\Rr^{1\times n}$ of $\vA$ corresponding to $\lambda$. Defining  $$\wlen{w}_\nu := \nu  \vectv_w$$ we   have the useful   $\wlen{\sigma(w)}_\nu=\nu  \vA \vectv_w  =\lambda \nu\vectv_w=\lambda\wlen{w}_\nu$.

\paragraph{Orbit graphs}\ 

Given a production system $(\aA, L, R)$, we construct for each orbit $(\omega^i,\anc_i)$, an \ddefine{orbit graph} with vertices  indexed by $i,j\in\Zz$,  labeled by each $\omega^i_j$. Edges connect the vertices $\omega^i_j$ with $\omega^i_{j'}$ and each $\omega^{i'}_j$ with  each $\omega^{i}_{\anc_{i}(j)}$. Clearly, every orbit graph is planar and the vertices are partitioned into ``rows", infinite paths corresponding to each $\omega^i$. As a planar graph, each face is   a cycle with vertices $\omega^i_{\anc_{i}(j)}$, $\omega^{i'}_j$, $\omega^{i'}_{j'}$, $\omega^i_{\anc_{i}(j')}$ --- each face is either a ``quadrilateral",  with four vertices,  or a ``triangle", with three, depending on whether $\anc_{i}(j') = (\anc_{i}(j))'$ or 
$\anc_{i}(j') = \anc_{i}(j)$. (Recall each $\anc_i$ is monotonic and onto.)

\begin{lem} Let $(\aA, \sigma)$ be an expansive primitive substitution system with an orbit $(\omega^i,\anc_i)$. Then any graph automorphism preserves the partition of the orbit graph $\Gamma$ into rows.\end{lem}

(The hypothesis that our regular production system is an expansive primitive substitution system is stronger than necessary, but we do need the existence of lots of triangles in the graph.)

\begin{proof} Because $(\aA, \sigma)$ is expansive, for some $a\in\aA$, $\wlen{\sigma(a)}\geq 2$. Because $(\aA, \sigma)$ is primitive, each $\omega^i$ has infinitely many occurrences of each $a\in\aA$ --- consequently, each row contains infinitely many edges which are the side of some triangle.

Given any quadrilateral with vertices $\omega^i_{\anc_{i}(j)}$, $\omega^{i'}_j$, $\omega^{i'}_{j'}$, $\omega^i_{\anc_{i}(j')}$, it is adjacent to a quadrilateral sharing the edge $\omega^{i'}_j$, $\omega^{i'}_{j'}$ on the $i'$th row, which is in turn adjacent to a quadrilateral sharing an edge on the $(i+2)$nd row, etc. Each quadrilateral thus lies within an infinite ``gallery" of quadrilaterals. 

Moreover,  any triangle $\Delta$ in an orbit graph is incident to  such a gallery, beginning with its ``row-edge". Since there are triangles in this row to the left and right of $\Delta$, $\Delta$ is incident to exactly one such gallery, and this gallery is transverse to the rows in the graph. Consequently, any graph isomorphism must preserve these galleries, and so preserve the rows. 
\end{proof}

\paragraph{Orbit tilings}\ 

We next define, at least for informal understanding, a kind of dual construction, due to L. Sadun~\cite{sadun_personal, gs_rptt}, of tilings of $\Hh^2$, corresponding to orbits in expansive primitive symbolic substitution systems. The construction is simplest to describe in the less used but quite convenient ``horocyclic model" of the hyperbolic plane, consisting of points $(i,j)\in\Rr^2$ with metric $(ds)^2= (e^{-y}dx)^2 + (dy)^2$. \mpp{Tiles shrink going down.---Cgs} One may easily check that maps of the form $(x,y)\mapsto (x+c, y)$ and $(x,y)\mapsto (e^{d} x, y+d)$ are isometries.  For familiarity, the  map 
$(x,y)\mapsto (x+(e^{y})i)$ isometrically takes the horocyclic model to the upper-half plane model.

Given an expansive primitive symbolic substitution system $(\aA,\sigma)$, with growth rate $\lambda$ and any choice of distribution $\nu$, for each $a\in\aA$, define $a$-tiles to be rectangles 
of Euclidean height $\log \lambda$ and width $
e^{d} \wlen{a}_\nu$ where the base of the rectangle lies on the line $y=d$. Clearly all $a$-tiles are congruent. 

We next note that because $\wlen{\sigma(a)}_\nu=e^{(\log \lambda)}\wlen{a}_\nu$, we may fit any $a$-tile directly above a row of tiles labeled in $\sigma(a)$, and more generally, any row labeled in a word $w$ may fit above a word labeled in $\sigma(w)$, each $w_i$ fitting above each $\sigma(w_i)$. 

More generally still, given any orbit $(\omega^i,\anc_i)$ we may form tilings of the entire plane. For precision, choosing any arbitrary $c,d\in\Rr$, we take a Euclidean rectangle corresponding to each $\omega^i_j$, with upper left corner at $(c+U^i_j+S_i, d- i\log \lambda)$ of Euclidean width $e^d \lambda^{-i} \wlen{\omega^i_j}_\nu$ and height $\log\lambda$,  where 
$$\textstyle U^i_j = e^d\lambda^{-i} \ \overset{j}{\underset k \sum}\   \wlen{u^i_k}_\nu \textrm{\ \ and \ \ }
S_i = e^d\ \overset{i}{\underset j  \sum} \ \left(\lambda^{-j} \overset{n_j}{\underset k \sum}\ \wlen{u^j_k}_\nu\right)  \textrm{\ \ where \ \ }n_i= \min \anc_j^{-1}(0)$$ ($U^i_j$ measures the Euclidean distance, in the model, of the left side of the tile corresponding to $\omega^i_j$ from the left side of the tile $\omega^i_0$. $S_i$ measures the horizontal distance from the left side of $\omega^i_0$ from the left side of $\omega^0_0$. The upper left corner of the tile corresponding to $\omega^0_0$ lies at the point $(c,d)$.)

\mpp{Should we add a lemma, that this actually is a tiling, with rows labeled in $\omega^i$, etc ---Cgs}
\mpp{Add pictures and examples --- Cgs}

\section{Main Technical Lemma}

Before stating the main technical lemma, for any fixed $N\in\Nn$, on any language, we take the relation $w \approxN w'$ on words $w,w'$ to mean that there exist (possibly empty) words $c,p,p',s,s'$ such that  
$w=pcs$, $w'=p'cs'$ and $N>\wlen{p},\wlen{p'},\wlen{s},\wlen{s'}$ --- that is, the words are equal apart from some prefix and suffix of length at most $N$. 

We say that  the growth rates $\lambda$, $\gamma$ of primitive expansive substitution systems $(\aA,\sigma)$, $(\bB,\varsigma)$ are \ddefine{incommensurate}  if for no $m,n\in\Nn$ is $\lambda^m=\gamma^n$. Equivalently, $\log \lambda / \log \gamma\notin \Qq$.  Substitution systems generically have incommensurate growth rates, and every substitution system will be incommensurate with at least one of the systems defined by ${\tt 0}\To {\tt 00}$ or ${\tt 0}\To {\tt 000}$.

\begin{lem}[Main Technical Lemma]
\label{Lemma:Technical Lemma}
Given primitive expansive substitution systems $(\aA,\sigma)$, $(\bB,\varsigma)$ there exist $K,N\in\Nn$ and 
a regular production system $(\aAb,L,R)$ equipped with maps $\alpha:\aAb\To\aA$, 
$\beta:\aAb\To\bB^\ast$, $\delta:\aAb\To\{K,K-1\}$ with the following properties:
\begin{enumerate}
\item For any orbit $\{(\omega^{i},\anc_{i})\}$ of $(\aAb,L,R)$,  $\{(\alpha(\omega^{i}), \anc_{i})\}$ is an orbit of $(\aA,\sigma)$.
\item  $\delta(w_{i})=\delta(w_{j})$ for all $w\in L$ and $1\leq i,j\leq \wlen{w}$. Consequently, $\delta$ is well-defined on $L$ and on $L^\infty$.
\item For all $w,w'\in L$ with $(w,w')\in R$,  $\beta(w')\approxN\varsigma^{\delta(w)}(\beta(w))$.
\item If the growth rates of $(\aA,\sigma),(\bB,\varsigma)$ are incommensurate, then no orbit of $(\aAb,L,R)$ is periodic.
\item $(\aAb,L,R)$ does have orbits.
\end{enumerate}
\end{lem}

\begin{proof} 
Let $\lambda, \gamma$ be the growth rates of $(\aA,\sigma),(\bB,\varsigma)$. Let $\nu\in\Rr^{|\aA|},\eta\in \Rr^{|\bB|}$ be  distributions for $(\aA,\sigma),(\bB,\varsigma)$, scaled so that $\nu_i>\gamma\ \eta_j$ for each $a_i\in\aA$, $b_j\in\bB$.  That is,  $\wlen{a}_\nu>\gamma\wlen{b}_\eta$ for all $a\in\aA,b\in\bB$. 

\psfrag{a}{$\alpha$}
\psfrag{b}{$\beta_1$}
\psfrag{c}{$\beta_{2\ldots}$}
\psfrag{d}{$b$}
\psfrag{e}[r][r]{$\varsigma^\delta(\beta b)=$}
\psfrag{p}{$p$}
\psfrag{q}{$q$}
\psfrag{r}{$r$}
\psfrag{s}{$s$}
\psfrag{t}{$t$}
\centerline{\includegraphics[height=2in]{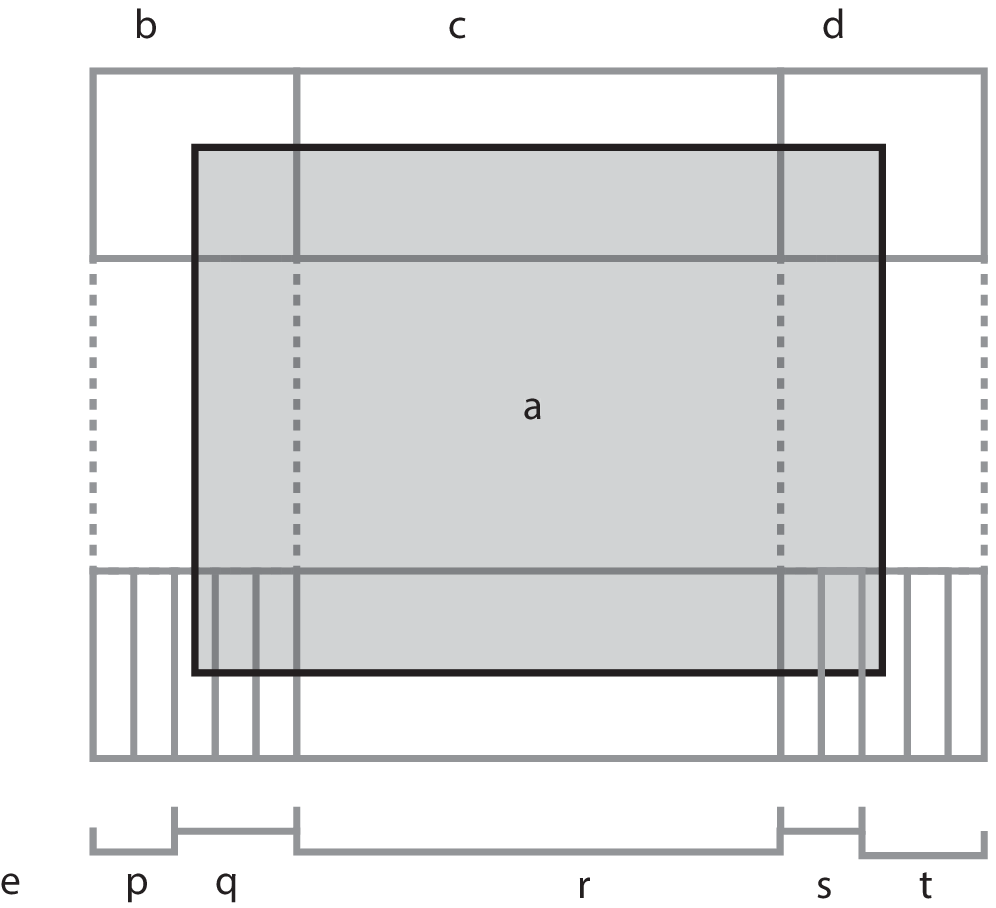}}

The basic idea is very simple: We overlay orbit tilings for $(\aA, \sigma)$ and $(\bB,\varsigma)$, reading off all possible ways that tiles in the latter can meet tiles in the former. Combinatorially, there are only finite many such possibilities, which we define as our alphabet $\aAb$ (choosing the relevant information we'll need). Our definition of $L$ captures   the correct (and rather simple) combinatorics along rows and our production rules $R$ describes  how each row may fit against the next locally. If we ignore the additional information we have placed on our tiles, we recover our original $\aA$-tiling--- that is, our $\aAb$-tiling is a marked $\aA$-tiling. And  we show that these additional markings do not allow any periodic tiling, simply because of the incompatibility of the growth rates, that is, the vertical spacing of the rows in the two tilings.

\vp
We may think of $\aAb$ as simply all possible ways for tiles in a $\bB$-tiling to cover a tile
in an $\aA$-tiling. Because $\wlen{a}_\nu>\gamma\wlen{b}_\eta$ for each $a\in\aA,b\in\bB$, each horizontal $\aA$ edge meets at least $3$ different $\bB$-tiles.
Considering a tile $T$ labeled $\alpha\in\aA$, with  width $\wlen{\alpha}_\nu$,   $\beta$  is any word corresponding to a string of $\bB$-tiles that cover the upper left corner of $T$, do not extend beyond the right of $T$, but with the addition of one more tile (labeled $b$) does extend beyond the right of $T$. The bottom of tile $T$ is $\delta=K$ or $(K-1)$ rows further below the row of $\beta$, and $\beta b$ is above a string of tiles $\varsigma^\delta(\beta b)$, which we partition into strings labeled $p,q,r,s,t$: $p$ is the maximal string to the left of $T$; $q$ is the rest of $\varsigma^{\delta}(\beta_1)$; $r = \varsigma^\delta (\beta_{2\ldots})$; and  $\varsigma^\delta(b)$ is partitioned into $s,t$ where $s$ is the maximal string entirely inside of $T$.  From $\alpha, \beta, \delta, p$ and $s$ any remaining combinatorics are completely implied.

Adjacent letters in any word in $L$ need only satisfy that the number of rows $\delta$ match, and that $s$ of a tile on the left matches $p$ of a tile on  the right. The production rules in $R$ only require that the sequence of $\beta$'s of the children match the string  $p q r$ of the parent. 

In this definition of $(\aAb,L,R)$, we simply take all possible such strings, constrained only by lengths measured in $\wlen{ }_\eta$ and $\wlen{ }_\nu$; properties (1)-(3) in the lemma are essentially trivial from the definition. It is more interesting that with no more care, just from the incompatibility of the growth rates of the two systems, that there are no periodic orbits. Finally, we must actually construct orbits --- we take care to do this explicitly, checking that our definitions are satisfied.

\paragraph{Defining $(\aAb,L,R)$, $\alpha, \beta, p,s, \delta$, $K$ and $N$}  \

Let $K:=\left\lceil \frac {\log \lambda}{\log \gamma}\right\rceil$.
We now define $\aAb$ as the set of all $(\alpha,\beta,p,s,\delta)\in\aA\times\bB^\ast\times\bB^\ast\times\bB^\ast\times\{K,K-1\}$ such that:
\begin{itemize}
\item For some $b\in\bB$, $\wlen{\beta_{(2\ldots)}}_\eta<\wlen{\alpha}_\nu<\gamma \wlen{\beta b}_\eta$
\item $\varsigma^\delta(\beta b) = p q r s t$ where $\varsigma^\delta(\beta_1) = p q$, $\varsigma^\delta(b) = s t$, and 
$\wlen{q_{(2\ldots)}r s}_\eta<\wlen{\sigma(\alpha)}_\nu<\gamma \wlen{q r s t_1}_\eta$
\end{itemize}
Note that $\aAb$ is finite and non-empty. For each $x=(\alpha_x,\beta_x,p_x,s_x,\delta_x)\in\aAb$,  define $\alpha(x)=\alpha_x, \beta(x)=\beta_x, p(x)=p_x, s(x)=s_x$ and $\delta(x)=\delta_x$. Take any $\displaystyle N>\max_{x\in\aAb} \wlen{p(x)}, \wlen{s(x)}$.

We define $L\subset \aAb^\ast$ by requiring that any length $2$ subword $w_{ii'}\in \aAb^\ast$ of any word $w$ in $L$ must satisfy  \begin{itemize}\item $s(w_i)=p(w_{i'})$ \item and  $\delta(w_{i})=\delta(w_{i'})$. \end{itemize}

We define production rules $R\subset \aAb\times L$ to be the pairs $(x,w)$ satisfying \begin{itemize} \item $p(x)\beta(w)=
\varsigma^{\delta(x)}(\beta(x))$ \item and $\sigma(\alpha(x))=\alpha(w)$ and extend  to production relations $R\subset L\times L, L^\infty\times L^\infty$.
\end{itemize}

\paragraph{Verifying  properties (1)-(3) of $(\aAb,L,R),\alpha,\beta,p,s$, $\delta$, $K$ and $N$.} 
\ 

 (1)  For any $(x,y)\in R$, by definition, we have
$\sigma(\alpha(x))=\alpha(y)$; thus for any orbit $\{(\omega^{i},\anc_{i})\}$ of $(\aAb,L,R)$,  $\{(\alpha(\omega^{i}),\anc_{i})\}$ is an orbit of $(\aA,\sigma)$.

(2) By the definition of $L$, $\delta$ is constant on the letters of any word in $L$.

(3) Let $w,v\in L$ with $(w,v)\in R$ and let $n=\wlen{w}$. Then $v=w^1\ldots w^{n}$ with $(w_k,w^k)\in R$ for all $1\leq k\leq n$.  
 
We will show by induction that
$$p(w_1) \beta(w^1)\ldots \beta(w^k)
=\varsigma^\delta(w)(\beta(w_{(1\ldots k)}) s(w_k)$$ for each $1\leq k\leq n$.

When $k=1$, this equality follows from the definition of $R$. Under the inductive hypothesis that
$$p(w_1)\beta(w^1)\ldots\beta(w^{k-1})
=\varsigma^{\delta(w)}(\beta(w_{(1\ldots(k-1))}))s(w_{(k-1)})$$
we have
$$\begin{array}{rcl}
p(w_1)\beta(w^1)\ldots\beta(w^k) 
& = & p(w_1)\beta(w^1)\ldots\beta(w^{k-1})\beta(w^k) \\ \\
& = & \varsigma^{\delta(w)}(\beta(w_{(1\ldots(k-1))}))s(w_{(k-1)}) \beta(w^k) \\ \\ 
& = &  \varsigma^{\delta(w)}(\beta(w_{(1\ldots(k-1))}))p(w_{k}) \beta(w^k) \\ \\ 
& = & \varsigma^{\delta(w)}(\beta(w_{(1\ldots(k-1))}))\varsigma^{\delta(w)}(\beta(w_k))s(w_k)\\ \\ 
& = &  \varsigma^{\delta(w)}(\beta(w_{(1\ldots k)}))s(w_k)\end{array}$$

for all $k$, and  in particular for $k=n$. Since  $N>\wlen{p(w_1)},\wlen{s(w_n)}$, 
$\beta(v) \approxN  \varsigma^{\delta(w)}(\beta(w))$.

\paragraph{(4) No orbit in $(\aAb,L,R)$ is periodic}

For contradiction, suppose $\{(\omega^{i},\anc_{i})\}\subset (L^\infty\times \Zz)^\Zz$ is an orbit of $(\aAb,L,R)$ with period $\pi$.

Let $u^0 = \omega^0(0)$. For all $k\in\Nn$, let $u^k=\omega^k(m\ldots M)$ where $m,M$ are the minimum and maximum of $\anc^{-1}_0 \circ\ldots \circ \anc^{-1}_{k-1} (0)$;  in other words, each $u^k$ is the 
 subword of $\omega^k$ with $(u^{k-1},u^k)\in R$, inductively, the ``$k$th descendant" of $u^0$.

Now $$\alpha(u^{k\pi})=\sigma^{k \pi} (\alpha(u^0)),$$
so that
$$\wlen{u^{k \pi}}=\Theta(\lambda^{k\pi}\wlen{u^0})=\Theta(\lambda^{k\pi}).$$

Let  $\Delta = \delta(u^0) + \ldots + \delta(u^{(\pi-1)})$. Because the orbit has period $\pi$, for all $k\in\Zz$, 
$\Delta = \delta(u^{k\pi}) + \ldots + \delta(u^{(\pi-1+k\pi)})$ and so 
$$\beta(u^{k\pi})\approxN\varsigma^{k\Delta}(\beta(u^0)),$$ 
and hence
$$\wlen{\beta(u^{k\pi})}=\Theta(\gamma^{k\Delta}\wlen{\beta(u^0)})=\Theta(\gamma^{k\Delta}).$$

On the other hand,  for all $u\in\aAb^\ast$, $\wlen{u}\leq \wlen{\beta(u)}\leq C \wlen{u}$ where $C= {\displaystyle \max_{a\in\aAb}}\wlen{\beta(a)}$, and so 
we have
$$\wlen{\beta(u^{k\pi})}=\Theta(\wlen{u^{k\pi}})
=\Theta(\lambda^{k\pi}).$$

But if $\Theta((\gamma^\Delta)^k)=\Theta((\lambda^\pi)^k)$, then we must have that $\lambda^{\pi}= \gamma^{\Delta}$, contradicting our hypothesis that $(\aA,\sigma)$, $(\bB,\varsigma)$ are incommensurate.

\paragraph{(5) Constructing an orbit in $(\aAb,L,R)$}
\

We have yet to show that there actually {\em are} any orbits in $(\aAb,L,R)$. 

In fact, we have great latitude:  let $\{(u^{i},\anc_{i})\}$ be an orbit of $(\aA,\sigma)$, let $\{(v^{i},\anct_{i})\}$ be an orbit of $(\bB,\varsigma)$, and let $c,d\in\Rr$, with which we may offset the orbit in $\bB$ against the orbit in $\aA$.  Given these parameters, we will construct letters $\omega_j^i$, show these  are correctly defined in $\aAb$, show that each $\omega^i=...\omega^i_{-1}\omega^i_0\omega^i_1\ldots$ is a word in $L^\infty$, and then show that each $(\omega^i,\omega^{i'})\in R$ relative to $\anc_i$. 

\paragraph{Defining $\delta_i$: the number of $\bB$ rows intersecting row $u^i$.} First we will establish some notation for measuring the orbits against one another; as discussed above, we may think of these as an  $\aA$-tiling overlaid with a $\bB$-tiling,  with an arbitrary offset of $c$ horizontally and $d$ vertically.
Recall that  $\lambda, \eta$ are the growth rates of $(\aA,\sigma)$ and $(\bB,\varsigma)$, and  $\nu, \eta$ are 
distributions scaled so that $\nu_a>\gamma\ \eta_b$ for each $a\in\aA$, $b\in\bB$.
For each $i,j\in\Zz$, let  $$\textstyle U^i_{j} :=\lambda^{-i} \ \overset{j}{\underset k \sum}\   \wlen{u^i_k}_\nu  $$ (We may think of $U^i_{j}$ as the distance to the left of  $u^i_j$ from the left of $u^i_0$, measured relative to the $\nu$, scaled to the horocyclic model.) 
Let  $$\Delta_{i}:=\lfloor (d+i\log\lambda)/\log\gamma\rfloor$$ --- that is, 
$\gamma^{\Delta_i} \leq  e^d\lambda^i  < \gamma^{(\Delta_i)'}$
--- and let $$\delta_i: =\Delta_{i'}-\Delta_{i}$$ 
(Thus the $\Delta_i$th row of the $\bB$ tiling contains the top edge of the $i$th row of the $\aA$ tiling, and $\delta_i$ is the number of rows in the $\bB$ tiling that meet the $i$th row, but not the bottom edge of the $i$th row, in the $\aA$-tiling.)

 Now $\delta_i =\Delta_{i'}-\Delta_{i}=\lfloor \frac{d+i\log\lambda}{\log\gamma}+\frac{\log\lambda}{\log\gamma}\rfloor- \lfloor \frac{d+i\log\lambda}{\log\gamma}\rfloor$
and so $\frac {\log \lambda}{\log \gamma}-1< \delta_i <\frac {\log \lambda}{\log \gamma}+1$.
Recalling that $K= \left\lceil \frac {\log \lambda}{\log \gamma}\right\rceil$, we have $\delta_i \in\{K,K-1\}$.

\paragraph{Defining $\nabla^i_j$: the index of the first $v^{\Delta_i}$ tile meeting $u^i_j$.} Next, there must exist some $\nabla^i_j\in \Zz$ satisfying 

$$\gamma^{-\Delta_i}\  \overset{\nabla^i_j}{\underset k \sum}\  \wlen{v^{\Delta_i}_k}_\eta 
\leq e^{-d}\ U^i_{j} + c \leq  \gamma^{-\Delta_i}\   \overset{(\nabla^i_j)'}{\underset k \sum}\  \wlen{v^{\Delta_i}_k}_\eta $$

(That is, $\nabla^i_j$ is the index of the letter in $v^{\Delta_i}$ just to the left, inclusively, of $u^i_j$, shifted by $d$ vertically and $c$ horizontally in the horocyclic model.)

Let $$V^i_j  :=  \gamma^{-\Delta_i}\  \overset{\nabla^i_j}{\underset k \sum}\  \wlen{v^{\Delta_i}_k}_\eta \textrm{\ \ and \ \ } W^i_j:=\gamma^{-\Delta_i}\   \overset{(\nabla^i_j)'}{\underset k \sum}\  \wlen{v^{\Delta_i}_k}_\eta$$
(We may think of $V^i_j$, $W^i_j$ as the distances to the left and right of $v^{\Delta_i}_{\nabla^i_j}$ from $v^{\Delta_i}_0$, scaled in the horocyclic model.)

Note that  $V^i_{j'}>W^i_j$ since  $\wlen{a}_\nu>\gamma\wlen{b}_\eta$ for all $a\in\aA,b\in\bB$. Moreover,  for all $k\geq j$, 
$$ \begin{array}{cccccccc} V^i_{k'}-W^i_j & \leq & e^{-d}(U^i_{k'}-U^i_j)  & < & W^i_{k'}-V^i_j \end{array}$$ and thus, using $\gamma^{\Delta_i}\leq e^{d}\lambda^{i}<\gamma^{\Delta'_i}$, we have
$$ \begin{array}{cccccccc} \gamma^{\Delta_i}(V^i_{k'}-W^i_j) & \leq & \lambda^i(U^i_{k'}-U^i_j)  & < & \gamma^{(\Delta_i)'}(W^i_{k'}-V^i_j) \end{array}$$ That is, 

\vp
\begin{equation}\label{UsefulEquation}\wlen{v^{\Delta_i}((\nabla^i_j)'\ldots \nabla^i_k)}_\eta \leq \wlen{u^i(j\ldots k)}_\nu<\gamma\wlen{v^{\Delta_i}(\nabla^i_j\ldots (\nabla^i_k)')}_\eta\end{equation}
as illustrated below for the $\Delta_i$th row of the $\beta$ tiling:

\vp

\psfrag{a}[c][c]{$v^{\Delta_i}((\nabla^i_j)'\ldots \nabla^i_k)$}
\psfrag{b}[c][c]{$u^i(j\ldots k)$}
\psfrag{c}[c][c]{$v^{\Delta_i}(\nabla^i_j\ldots (\nabla^i_k)')$}
\centerline{\includegraphics{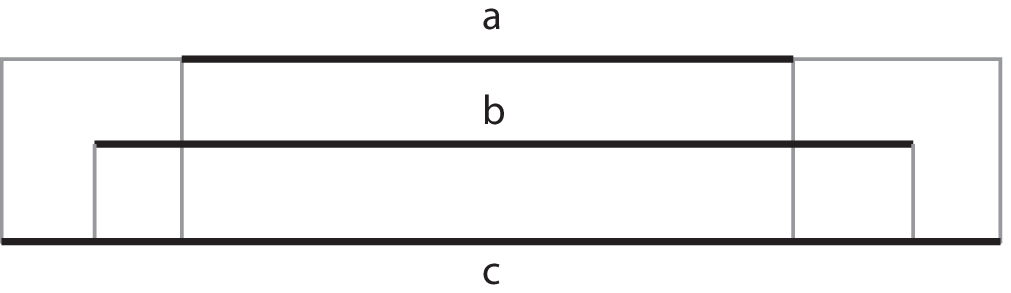}}

\vp
\paragraph{Defining $\omega^i_j$.} For each $i,j\in\Zz$   we now define $\omega^i_j$, giving $\alpha(\omega^i_j), \beta(\omega^i_j), p(\omega^i_j)$, $s(\omega^i_j)=p(\omega^i_{j'})$ (and letting $\delta(\omega^i_j)= \delta_i$). We will show each satisfies the definition of $\aAb$, that each $\omega^i(jj')\in L$ and so each $\omega^i\in L^\infty$. We will show each  $(\omega^i_j,\omega^{i'}_{\anc_i^{-1}(j)})$ satisfies the definition of $R$ and so too then $(\omega^i,\omega^{i'})\in R$ with respect to $\anc_i$--- thus $(\omega^i,\anc_i)$ is an orbit of $(\aAb,L,R)$, concluding the construction.

For each $i,j\in\Zz$, we define $$\alpha(\omega^i_j):=u^i_j\in\aA$$ 
Define $$\beta(\omega^i_j):=v^{\Delta_i}({{\nabla^i_{j}}\ldots (\nabla^i_{j'}-1)})$$ 
Thus, taking $k=j$ in equation~(\ref{UsefulEquation}) above, 
$$\wlen{\beta(\omega^i_j)_{(2\ldots)}}_\eta < \wlen{\alpha(\omega^i_j)}_\nu < \gamma \wlen{\beta(\omega^i_j)\beta(\omega^i_{j'})_1}_\eta
$$ satisfying the conditions on $\beta$.

Let $\anct_i^{-\delta_i} := \anct_i^{-1} \circ \ldots \circ \anct_{(i+\delta_i-1)}^{-1}$ 
and define  $m^i_j, n^i_j$ and $M^i_j$ as:

$$m^i_j := \min \anct^{-\delta_i}_{\Delta_i}(\nabla^i_j)$$ 
$$M^i_j := \max \anct^{-\delta_i}_{\Delta_i}(\nabla^i_j)$$
the indices of the first and last letters of $\varsigma^{\delta_i}(v^{\Delta_i}_{\nabla^i_j})$ in $v^{\Delta_{i'}}$, and
$$n^i_j := \min \anc_i^{-1}(j)$$ the index of the first letter of $\sigma(u^i_j)$ in $u^{i'}$ --- note $m^i_j\leq  \nabla^{i'}_{n^i_j} <  M^i_j < m^i_{j'}$. Moreover 
$$\varsigma^{\delta_i}(\beta(\omega^i_j))=v^{i'}(m^i_j\ldots \nabla^{i'}_{n^i_j} \ldots M^i_j \ldots (m^i_{j'}-1))$$
and in particular 
$$\varsigma^{\delta_i}(\beta(\omega^i_j)_1)=v^{i'}(m^i_j\ldots M^i_j )$$
If $\nabla^{i'}_{n^i_j}\gneq m^i_j$ let 
$$p(\omega^i_j)= v^{\Delta_{i'}}({m^i_j}\ldots (\nabla^{i'}_{n^i_j}-1))$$ and 
let $p(\omega^i_j)$ be the empty word otherwise. 

 Letting 
 $p=p(\omega^i_j)$, 
 $q=v^{\Delta^{i'}}(\nabla^{i'}_{n^i_j}\ldots M^i_j)$, 
 $r = v^{\Delta^{i'}}((M^i_j)'\ldots (m^i_{j'}-1))$, 
 $s=s(\omega^i_j) = p(\omega^i_{j'})$, and 
 $t = v^{\Delta^{i'}}(\nabla^{i'}_{n^i_{j'}}\ldots M^i_{j'})$ 
 we now have 
$$\varsigma^{\delta_i}(\beta^{ij}\beta^{ij'}_1) = p q r s t, 
\textrm{\ \ \ } 
\varsigma^\delta(\beta^{ij}_1) = p q, 
\textrm{\ \ and \ \ } 
\varsigma^\delta(\beta^{ij'}_1) = s t$$
Since $\wlen{\sigma(\alpha^{ij})}_\nu = \wlen{u^{i'}(n^i_j\ldots (n^i_{j'}-1))}_\nu$, applying equation~(\ref{UsefulEquation}) we have 
$$\wlen{q_{(2\ldots)}r s}_\eta<\wlen{\sigma(\alpha^{ij})}_\nu<\gamma \wlen{q r s t_1}_\eta$$

Consequently, each $\omega^i_j$ is a well-defined letter in $\aA$. Moreover, as each $\delta(\omega^i_j)=\delta_i=\delta(\omega^i_j)$ and each $p(\omega^i_{j'})=s(\omega^i_j)$, we have that 
every $\omega^i(jj')\in L$ and every $\omega^i\in L^\infty$. Finally,  
each $(\omega^i_j,\omega^{i'}(\anc^{-1}_i(j)))$ is a production rule in $R$ since 
$p(\omega^i_j)\beta(\omega^{i'}(\anc^{-1}_i(j))) = \varsigma^{\delta_i}(\beta(\omega^i_j))$ and $\sigma(\alpha(\omega^i_j)) = \alpha(\omega^{i'}(\anc^{-1}_i(j)))$.

\end{proof}

We do not need and will not prove, but in fact

\begin{thm} Every orbit of $(\aAb,L,R)$ is of the form constructed in the proof above.\end{thm}

\section{Strongly aperiodic subshifts of finite type on surface groups}
We will now define a strongly aperiodic subshift of finite type on a hyperbolic surface group. The idea is that there is a primitive symbolic substitution system which in some sense describes the combinatorics of the tiling of $\mathbb{H}^{2}$ by the fundamental domain of a surface group. Applying Lemma \ref{Lemma:Technical Lemma} will then allow us to produce a strongly aperiodic subshift of finite type on the dual graph of this tiling, which is just the Cayley graph of the surface group.

\subsection{The $\{p,q\}$-graph, periods, and subshifts of finite type.}

The $\{p,q\}$-graph $\Gpq$ is the unique planar graph such that every vertex has degree $q$ and every complementary region (in the plane) is bounded by a $p$-cycle. When $p=q=4g$, this $\Gpq$ will be isomorphic to an unmarked unlabeled Cayley graph $\Gsurf$ of the genus $g$ surface group. 

We will extend the notion of a subshift of finite type from Cayley graphs to graphs of bounded degree. In fact, our definition will be slightly more restrictive, since an arbitrary graph carries no natural framing on its vertices. After discussing what it means for such an object to be strongly aperiodic, we will define a primitive substitution system whose orbit graphs each contain a naturally embedded copy of the $\{p,q\}$-graph which fills out every vertex. This substitution system will be critical in what follows.

\paragraph{Labeled graphs and subshifts of finite type.} Given an alphabet $\aA$, an \ddefine{$\aA$-labeling }  of a graph $\Gamma$ with vertices $\vertices(\Gamma)$ is a map $\labelmap\in\aA^{\vertices(\Gamma)}$. An \ddefine{ $\aA$-labeled graph} $(\Gamma,\labelmap)$ is a graph $\Gamma$ together with an $\aA$-labeling $\labelmap$ of it. 
An \ddefine{ $\aA$-pattern} is a  finite $\aA$-labeled graph $(\Delta,\labelmap)$ with a specified basepoint $V_0(\Delta)\in\vertices(\Delta)$. 

A \ddefine {labeled isomorphism} $\Phi:(\Gamma_1,\labelmap_1)\to (\Gamma_2,\labelmap_2)$, with $\labelmap_1$ an $\aA$-labeling and $\labelmap_2$ a $\bB$-labeling on alphabets $\aA$ and $\bB$,
 consists of a  graph isomorphism $\Phi:\Gamma_1\to\Gamma_2$, together with a map $\Phi^{*}:\aA\to\bB$, such that $\labelmap_2(\Phi(v)):=\Phi^*(\labelmap_1(v))$. If $\aA=\bB$ then $\Phi^*$ is typically understood to be the identity. 

Given a graph $\Gamma$ and a finite collection $\fF$ of finite $\aA$-labeled graphs, 
a \ddefine{subshift of finite type} $X(\Gamma,\fF)$  on $\Gamma$  consists of all 
$\aA$-labelings $\labelmap$ of $\Gamma$, such that every vertex $v$ 
of $\Gamma$ lies within an  subgraph $\Gamma'$ so that there is a labeled isomorphism $\Phi$ from  
 $(\Gamma',\labelmap|_{\Gamma'})$ to some $(\Delta, \labelmap')\in\fF$, with 
  $\Phi(v) = V_0(\Delta)$. That is,  in effect, every vertex is the basepoint of some pattern specified in $\fF$.

Given $\labelmap\in X(\Gamma,\fF)$, a period of $\labelmap$ is any labeled isomorphism $\Phi:(\Gamma,\labelmap)\To(\Gamma,\labelmap)$.  We say that $X(\Gamma,\fF)$ is strongly aperiodic if no labeling in $X$ has a nontrivial period.

\paragraph{Cayley graphs.} Note that if $\Gamma$ is the (unlabeled) Cayley graph of a group $G$, then every subshift of finite type $X$ on $\Gamma$ induces a subshift of finite type $\tilde{X}$ on $G$, and any period of $\tilde{X}$ induces a period of $X$. However, a subshift of finite type on $G$ will not always be a subshift of finite type on $\Gamma$ because forbidden patterns on $G$ can use framing data which is not available on $\Gamma$.

\paragraph{A primitive substitution system encoding the $\{p,q\}$-graph.} 

For the rest of this section, fix $p,q\geq 5$, 
 fix $\aA=\{\tX,\tY\}$ and fix
$$\sigma(\tX)=(\tX\tY^{p-3})^{q-4}\tX\tY^{p-4},$$
$$\sigma(\tY)=(\tX\tY^{p-3})^{q-3}\tX\tY^{p-4}.$$

Given a production system $(\aA, L, R)$,  for each orbit $(\omega^i,\anc_i)$, 
we construct a \ddefine{reduced orbit graph}  $\reducedorbitgraph 
$  with vertices  indexed by $i,j\in\Zz$,  labeled by each $\omega^i_j$. Edges connect the vertices $\omega^i_j$ with $\omega^i_{j'}$ and each $\omega^{i'}_j=\tX$
(rather than all $\omega^{i'}_j$)
with  each $\omega^{i}_{\anc_{i}(j)}$.  (That is, a reduced orbit graph corresponds to an orbit graph with the edges produced by $\tY$ vertices removed.) More generally, if $\orbit$ is an orbit of some RPS $(\aAO,L,R)$, then for any function $\mu:\aAO\To\aA$, we can define a corresponding reduced orbit graph where we remove vertical edges produced by $\mu^{-1}\{\tY\}$ vertices

\psfrag{a}[l][l]{$\tX\tY$}
\psfrag{b}[l][l]{$\tX\tY\tY\tX\tY\tY\tX\tY\ \tX \tY\tY\tX \tY\tY \tX \tY\tY \tX\tY$}
\centerline{\includegraphics{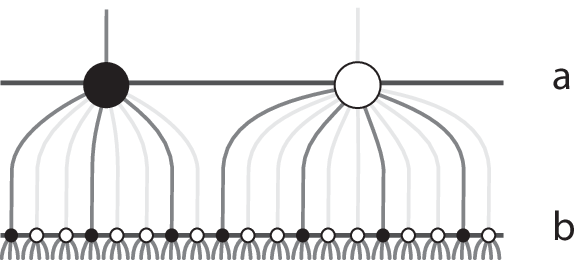}}

The following is easily verified and further discussion appears in~\cite{gs_stap}; the construction for the $\{5,6\}$ graph is shown above.

\begin{lem}
For all orbits $\orbit$ of $(\aA,\sigma)$, the reduced orbit graph $\reducedorbitgraph$ has underlying graph $\Gpq$.
\end{lem}

If $\bB$ is any other primitive symbolic substitution system, then for any orbit $\orbit=(\omega^i,\anc_i)$ of the RPS $(\aAb,L,R)$ given by Lemma \ref{Lemma:Technical Lemma}, it follows that $\reducedorbitgraph$ has underlying graph $\Gpq$ (where we reduce with respect to $\alpha:\aAb\To\aA$ as in the statement of the Lemma). This is because of the first property promised by the Lemma---that $(\alpha(\omega^i),\anc_i)$ is an orbit of $(\aA,\sigma)$.

\subsection{A strongly aperiodic subshift of finite type on the $\{p,q\}$-graph.} 
Fix $(\bB,\varsigma)$ with growth rate incommensurable with that of $(\aA,\sigma)$ --- in fact taking $\bB=\{{\mathtt 0}\}$ and $\varsigma({\mathtt 0})={\mathtt 0}^k$ for any fixed $k>1$ will suffice --- and let $(\aAb, R,L)$ be as in Lemma \ref{Lemma:Technical Lemma}. We would like to say that the collection of $(\aAb,R,L)$ orbit graphs forms a subshift of finite type on $\Gpq$, and that this subshift of finite type is strongly aperiodic, but for technical reasons, we must adjust this plan as follows.

First we construct a primitive substitution system $(\Asharp,\sigmasharp)$ and a SFT $\XF$ on $\Gpq$ such that every configuration of $\XF$ is labeled isomorphic to a reduced orbit graph of some orbit $\orbit$ of $(\Asharp,\sigmasharp)$. In fact the system $(\Asharp,\sigmasharp)$ will be essentially the same as $(\aA,\sigma)$, but the alphabet $\Asharp$ will carry some extra data that allows us to recover the orbit graph structure of $\reducedorbitgraph$ knowing only the labels $\omega_\orbit:\Gpq\To\Asharp$. We then modify $(\aAb,R,L)$ to produce an RPS $(\aAO,R_0,L_0)$ and a SFT on $\Gpq$ whose configurations are isomorphic to $\aAO$ orbit graphs, and hence have no periods.

\paragraph{Defining $(\Asharp,\sigmasharp)$.} Let $N$ be $\wlen{\sigma(\tY)}$.

\begin{definition} Let $\Asharp=\aA\times\{1,\ldots N\}$, and for $a\in \aA$, let
$$\sigmasharp(a,i)=(\sigma(a)_1,1)(\sigma(a)_2,2)\ldots
(\sigma(a)_{\wlen{\sigma(a)}},\wlen{\sigma(a)}).$$
I.e., $\sigmasharp(a,i)$ ignores $i$ and writes out the letters of $\sigma(a)$, labeled by their place within the word.
\end{definition}

For a vertex $v\in\Gpq$, let $\Delta_v$ be the union of $p$-cycles containing $v$. Let $\Delta_{p,q}$ be an abstract graph isomorphic to $\Delta_v$.

\begin{lem}
Let $\fF$ be the collection of all patterns realized on $\Delta_v$ for $v$ any vertex of any $(\Asharp,\sigmasharp)$ reduced orbit graph.

Then for any $\omega\in \XF$, there exists an orbit $\orbit$ of $(\Asharp,\sigmasharp)$ such that
$$(\Gpq,\omega)\cong(\reducedorbitgraph,\omega_\orbit).$$
\end{lem}

\begin{proof}
Before outlining the proof, we need to make an observation about $p$-cells in $\Asharp$ orbit graphs.

\begin{definition}
For $a\in A$, let $\fraks_a=\{i:\sigma(a)_i=\tX\}$. Given a sequence $n_1,\ldots,n_k$ of natural numbers, we say that this sequence is of horizontal type if the following hold.
\begin{itemize}
\item The $n_i$ are consecutive, i.e., $n_{i'}=(n_i)'$.
\item $k=p-1$ or $p-2$.
\item For $\wlen{\sigma(a)}=k$, we have $n_1$ and $n_k$ in $\fraks_a$.
\end{itemize}
\end{definition}

\paragraph{Key observation.} Letting $\Gamma_p$ be an abstract $p$-cycle,  let $\orbit$ be an orbit of $(\Asharp,\sigmasharp)$ and $\alpha:\Gamma_p\rightarrow \reducedorbitgraph$ a $p$-cycle in $\reducedorbitgraph$. Then we have the following.
\begin{itemize}
\item There exists {\it exactly} one oriented path $\gamma:\{1,\ldots,k\}\To \Gamma_p$ whose numerical labels $\omega_\orbit\circ\alpha\circ\gamma$ form a sequence of horizontal type.
\item If $\gamma$ has $p-1$ vertices, then all the vertices of this path are produced by the other vertex of the $p$-cycle.
\item Suppose $\gamma$ has $p-2$ vertices, $v_0$ is the vertex preceding $\gamma$ in cyclic order, and $v_1$ the vertex following $\gamma$ in cyclic order. Then the first $p-3$ vertices of $\gamma$ constitute the last (rightmost) $p-3$ vertices produced by $v_0$, and the last vertex of $\gamma$ is the first vertex produced by $v_1$.
\end{itemize}
It is easy to see that there is at least one such path as one of the vertices of the $p$-cycle must produce at least $p-3$ of the other vertices. On the other hand, it is clear that there is at most one such path with length $p-1$ and at most one with length $p-2$, as $i\in\fraks_a$ implies that $i'$ and $i+2$ are not in $\fraks_a$. We leave the rest to the reader.

Now we may proceed to outline the proof of the lemma. We will construct a bijection $\psi:\vertices(\Gpq)\To\Zz^2$ and $P:\Gpq\To\Zz$ such that the following hold.
\begin{itemize}
\item The function $\row_i:\Zz\To\Gpq$ defined by $\row_i(j)=\psi^{-1}(j,i)$ yields an infinite path in $\Gpq$.
\item The function $P_i:\Zz\To\Zz$ given by $P_i(j)=P(\row_{i'}(j))$ is a nondecreasing surjection.
\item $\omega|_{\row_{i'}(P_i^{-1}(j))}$ is the same word as $\sigmasharp(\omega(\psi^{-1}(j,i))$.
\item The neighbors of $\psi^{-1}(j,i)$ in $\row_{i'}\Zz$ are exactly the $\tX$-labeled vertices of $\row_{i'}(P_i^{-1}(j))$.
\end{itemize}

Then we can define $\orbit=(\omega^i,\anc_i)$ by taking $\anc_i=P_i$ and $\omega^i=\omega\circ\row_i$.

\paragraph{Defining the second coordinate of $\psi$.} We will define a function $\dy$ on oriented edges of $\Gpq$ as follows. Every edge $e$ participates in two $p$-cycles. If, in either cycle, it is part of the path whose numerical labels are of horizontal type, then we set $\dy(e)=0$. Otherwise, we set $\dy(e)=1$ if the terminal vertex of $e$ lies on this path, and $\dy(e)=-1$ if the initial vertex of $e$ lies on this path. (This exhausts all cases: if neither vertex of $e$ lies in the horizontal type path, then $e$ is horizontal in the other $p$-cycle).

Now, fix a vertex $v_0$ of $\Gpq$. Define $y(v)$ to be the sum over any path $\gamma$ from $v_0$ to $v$ of $\dy$ of the edges of $\gamma$. Then $y$ is well defined because it is easily seen that $\dy$ sums to $0$ around any $p$-cycle, and every loop can be broken into $p$-cycles. This $y$ will give the second coordinate of $\psi:\Gpq\To \Zz\times \Zz$.

\paragraph{Horizontal and vertical edges.} Let $\HG$ be the subset of $\Gpq$ given by throwing away the interior of any edge $e$ such that $\dy(e)\neq 0$. (We say that the edges of $\HG$ are horizontal, while the other edges of $\Gpq$ are vertical). Of course, $y$ is constant on each connected component of $\HG$. We wish to show that $y^{-1}(i)$ forms a single connected component of $\HG$ and that this connected component is an infinite embedded path.

It follows from our key observation that every $v\in \Gpq$ is incident to exactly two horizontal edges, which carry a natural orientation. Hence, all connected components of $\HG$ are either cycles or lines. Furthermore, $v$ has a natural parent $\Parent(v)$, the unique vertex connected to $v$ by a vertical edge with $v$ at the lower end, in $\Delta_v$.

\begin{prop}
Given $u,v\in\vertices(\Gpq)$ we have that $u,v$ are in the same connected component of $\HG$ if and only if $\Parent(u)$ and $\Parent(v)$ are in the same connected component of $\HG$.
\end{prop}

\begin{proof}
The parent of a horizontal neighbor of a vertex $v$ is either equal to the parent of $v$ or is a horizontal neighbor of the parent of $v$. Hence, if $u,v$ are in the same connected component, so are their parents.

Conversely, if $u$ and $v$ are horizontal neighbors, and $u$ is the parent of some vertex $\tilde{u}$ and $v$ the parent of $\tilde{v}$, it is easily seen (in $\Delta_v$) that there is a horizontal path connecting $\tilde{u}$ and $\tilde{v}$. Hence, if two vertices have parents in the same connected component, they must themselves be in the same connected component.
\end{proof}

Consequently, the collection of parents of vertices of some connected component itself is the vertex set of some (nonempty) connected component. If the original connected component is finite, it is clear (from our observation) that the set of parents is strictly smaller. Taking a minimal connected component, we obtain a contradiction. Hence, every connected component is infinite, and it follows that every connected component is a line.

\paragraph{Connectedness of $y^{-1}(i)$.} We wish now to show that $y^{-1}(i)$ consists of exactly one infinite embedded path. Suppose $y^{-1}(i)$ includes two connected components of $\HG$. Let $\gamma$ be a path connecting these two, such that $\gamma$ includes a minimal number of vertical edges. We know $\gamma$ (or its reverse) must contain at least one vertical edge, hence it contains a subpath consisting of a positive vertical edge, some horizontal edges, and a negative vertical edge.  By the proposition, we can replace this subpath with a horizontal subpath, contradicting minimality.

\paragraph{Defining the first coordinate of $\psi$.} We have seen that each $y^{-1}(i)$ is an infinite horizontal line. For each $i$, choose any vertex $v_0$  in $y^{-1}(i)$, and define $\psi|_{y^{-1}(i)}(v)$ to be the signed count of horizontal edges between $v_0$ and $v$.

\paragraph{Defining $P$.} Let $P(v)$ be the first coordinate of $\psi$ of the parent of $v$. By the proposition, every vertex of $\row_i$ is the parent of some vertex in $\row_{i'}$, so $P_i$ is surjective. It is clear that $P_i$ is nondecreasing, since if $u$ is immediately to the right of $v$, then the parent of $u$ is either immediately to the right of the parent of $v$ or they coincide.

\paragraph{Verifying properties.} For each vertex $v$, the vertices produced by $v$ in $\Delta_v$ form a path labeled by $\sigmasharp(\omega(v))$. Since these are exactly the vertices whose parent is $v$, all have $y$ value one greater than that of $v$, and the path is horizontal, we are done.
\end{proof}

Now, as promised, take $(\bB,\varsigma)$ with incommensurable growth rate.  Let $\aAO=(\Asharp)_\bB$, and let $R$, $L$ be as in the lemma.

\begin{cor}\label{Cor:MainThm}
Let $\fF$ consist of all $\aAO$-patterns $\omega$ on $\Deltapq$ such that the underlying $\Asharp$ labels appear in some $(\Asharp,\sigmasharp)$ orbit graph, and if $\gamma\subset\Deltapq$ is the path of vertices produced by $v$ in this orbit graph, then $\omega|_\gamma\in L$ and $(\omega(v),\omega|_\gamma)\in R$.

Then every configuration of $\XF$ is isomorphic to an $\aAO$ orbit graph, and hence $\XF$ is strongly aperiodic.
\end{cor}

\begin{proof}
The orbit graph structure is the same as in the lemma. We must only verify that the rows are admissible and that each row produces the next. But this immediately follows from the definition of $\fF$.
\end{proof}

By applying the corollary with $p=q=4g$, we see that every higher genus surface group carries a strongly aperiodic subshift of finite type.

\normalsize

\end{document}